\documentclass[a4paper,11pt]{amsart}

\usepackage{docmute}
\usepackage{mathabx}
\usepackage{url}

\theoremstyle{plain}
\newtheorem{thm}{Theorem}[section]
\newtheorem{prop}[thm]{Proposition}
\newtheorem{lem}[thm]{Lemma}
\newtheorem{cor}[thm]{Corollary}

\theoremstyle{definition} 
\newtheorem{dfn}[thm]{Definition}
\newtheorem{note}[thm]{Notation}

\theoremstyle{remark}
\newtheorem{rem}[thm]{Remark}

\newcommand{\N}{\mathbb{N}}
\newcommand{\R}{\mathbb{R}}

\newcommand{\pa}{\mathrm{PA}}
\newcommand{\fixn}{\mathrm{Fix}_\mathrm{N}}

\newcommand{\out}[1]{\mathrm{Out}(F_{#1})}

\newcommand{\auto}{\mathrm{Aut}(F)}
\renewcommand{\outer}{\mathrm{Out}(F)}
\newcommand{\fix}[1]{\mathrm{Fix}({#1})}
\newcommand{\nfix}[1]{\mathrm{Fix_N({#1})}}
\newcommand{\wt}{\widetilde}
\newcommand{\OUT}{\mathrm{Out}(F)}
\newcommand{\IT}{\mathrm{Isom}(T_{+})}

\begin{document}
\title{Fibered commensurability on $\mathrm{Out}(F_{n})$}

\author{Hidetoshi Masai}
\address{Department of Mathematics Tokyo Institute of Technology 2-12-1, Ookayama, Meguro-ku, Tokyo. 152-8551. Japan}
\email{masai@math.titech.ac.jp}

\author{Ryosuke Mineyama}
\address{Department of Mathematics, Osaka City University, Suimiyoshi, Osaka 558--8585, Japan}
\email{mineyama@sci.osaka-cu.ac.jp}

\maketitle

\begin{abstract}
We define and discuss a notion called fibered commensurability of outer automorphisms of free groups.
This notion lets us study symmetry of outer automorphisms.
The notion of fibered commensurability is first defined by Calegari-Sun-Wang on mapping class groups.
The Nielsen-Thurston type of mapping classes is a commensurability invariant.
One of the important facts of fibered commensurability on mapping class groups is for the case of pseudo-Anosovs,
there is a unique minimal element in each fibered commensurability class.
For outer automorphisms, we first show that being atoroidal and fully irreducible is a commensurability invariant. 
Then for such outer automorphisms, we prove that there is a unique minimal element in each fibered commensurability class,
under a certain asymmetry condition on the ideal Whitehead graphs.
\end{abstract}

\section{Introduction}
We define {\em fibered commensurability} of outer automorphisms of free groups.
Fibered commensurability is defined by a natural covering relation between outer automorphisms, which allows us to discuss symmetry of outer automorphisms.
First, we recall the definition of fibered commensurability of mapping classes on surfaces introduced by Calegari-Sun-Wang \cite{CSW}.
A mapping class $\phi_1$ on a compact surface $S_1$ is said to cover another mapping class $\phi_2$ on a compact surface $S_2$ if $\phi_1$ is a lift of a power $\phi_2^n$ with respect to a finite covering $p:S_1\rightarrow S_2$.
Two mapping classes are (fibered) commensurable if there is a third mapping class which covers both.
We may rephrase this covering relation in terms of the action on the fundamental groups;
$\phi_1$ covers $\phi_2$ if we can find a finite index copy of $\pi_1(S_1)$ in $\pi_1(S_2)$ and representatives $\Phi_{1}, \Phi_{2}$ of $\phi_{1},\phi_{2}$ respectively so that $(\Phi_2)^n_*|_{\pi_1(S_1)} = (\Phi_1)_*$ for some $n\in\mathbb{N}$.
This equivalent definition of coverings can be adapted for the case of outer automorphisms of free groups of rank $\geq 2$.
One subtle difference is that we need to pass to powers to define commensurability as an equivalence relation, 
see Section \ref{sec.fcms} for a detail.
In Section \ref{sec.fcms}, we define an equivalence relation called covering equivalence in order to make each commensurability class a partially 
ordered set.
For the case of surfaces, it can be readily seen that the Nielsen-Thurston type is a commensurability invariant, see \cite{CSW}.
It is well-known that there are several similarities between pseudo-Anosov mapping classes and fully irreducible outer automorphisms.
However, in the case of free groups, fully irreducible elements may be covered by reducible elements.
Indeed, given any {\em geometric} fully irreducible outer automorphism $\phi$, 
any lift of $\phi$ corresponding to a mapping class of a surface with more than one boundary components would have a reducible power.
This is because such a lift has a power which fixes the conjugacy class of a free factor subgroup corresponding to the boundary.
The first aim of this paper is to show that for the non-geometric case, fully irreducibility is a commensurability invariant property.

It is proved in \cite{CSW,Mas13,Mas17} that for the case of pseudo-Anosov mapping classes,
each commensurability class contains a unique minimal (orbifold) element.
Moreover, in \cite{Mas13}, it is proved that for any pseudo-Anosov mapping class whose 
(un)stable foliation has singularities only on punctures,
the minimal element must be defined on a surface (i.e. no orbifold singular point exists).
The existence of the unique minimal element is used for example, 
to prove that cusped random mapping tori are non-arithmetic in \cite{MasGGD}.
In this paper, instead of defining ``orbifolds'', we consider the case where fully irreducible elements are {\em atoroidal} and each component of the ideal Whitehead graph admits no symmetry.
The definition of atoroidal outer automorphisms and the ideal Whitehead graphs are recalled in Section \ref{sec.pre}.
Since every element $\phi\in\out{2}$ is geometric, we focus on $F_{n}$ with $n\geq 3$.
Then our main theorem is the following.
\begin{thm}\label{thm.main}
Let $\phi$ be an outer automorphism of a free group of rank $\geq 3$.
Suppose $\phi$ is 
\begin{itemize}
\item atoroidal, fully irreducible, and
\item every component of the ideal Whitehead graph admits no symmetry. 
\end{itemize}
Then the commensurability class of $\phi$ contains a unique minimal element (or more precisely, covering equivalence class) which is an outer automorphism of a free group of rank $\geq 3$.
\end{thm}

\begin{rem}
In \cite{Pfa}, Pfaff give a method to construct atoroidal 
fully irreducible outer automorphisms whose ideal Whitehead graph is given by ``pasting'' two ideal Whitehead graphs of other outer automorphisms.
By using her result, we can construct examples of outer automorphisms satisfying the condition of Theorem \ref{thm.main}.
\end{rem}

This paper is organized as follows.
In Section \ref{sec.pre}, we prepare basic facts of $\out{n}$.
We define fibered commensurability in Section \ref{sec.fcms}, and there,
we prove that it is an equivalence relation.
Several basic properties of fibered commensurability on $\out{n}$ are also discussed in Section \ref{sec.fcms}.
Then in Section \ref{sec.minimal}, we prove the main theorem.


\section{Preliminaries}\label{sec.pre}
We refer the reader to the book of Handel and Mosher \cite{HM11} for the topics discussed in this section.
Most of our terminologies are defined along Chapter 2 of the book.

\subsection{Marked graphs}
In this article, unless otherwise stated, a \emph{graph} $G$ is always a finite cell complex of dimension $1$ 
with all vertices having valence greater than or equal to $2$.
We denote by $V(G)$, $E(G)$ the set of vertices and the set of edges respectively.
A \emph{metric graph} is a graph $G$ equipped with a path metric defined 
by a function $\ell: E(G) \to (0, \infty)$.
Here, a {\em path} $\gamma$ is a concatenation of edges $\gamma = e_1e_2\cdots e_n$ ($e_i \in E(G)$).
If the initial point of a path $\gamma$ coincides with the endpoint of $\gamma$ then it is called a \emph{loop}.
A \emph{geodesic} on a graph is a locally shortest path parametrized by arc length.
Any finite path $\gamma$ is homotopic rel endpoints to a unique geodesic; 
such a geodesic is said to be \emph{obtained from $\gamma$ by tightening} and denoted by $\gamma_{\sharp}$. 
An $\mathbb{R}$-tree is a metric space $T$ such that for any distinct two points $x, y\in T$,
 there exists a unique embedded topological arc, 
 denoted $[x,y]\subset T$, with endpoints $x,y$ and of length $d(x,y)$.
We call such an embedded arc a \emph{geodesic} on the $\R$-tree $T$.
An $\mathbb{R}$-tree is called \emph{simplicial} if the set of vertices forms a discrete set.
If an $\mathbb{R}$-tree admits an isometric action by a free group $F$, we call it an $F$-tree.
The universal covering $\wt{G}$ of a finite metric graph $G$ is a simplicial $F$-tree where $F\cong \pi_{1}(G)$.

The $n$-\emph{rose} $R_n$ is a graph defined as the wedge of $n$ circles. 
We call a homotopy equivalence $\rho: R_n \to G$ from the $n$-rose to a metric graph $G$ a \emph{marking}.
A \emph{marked graph} is a pair $(G, \rho)$ of a metric graph and a marking.


Given graphs $G, G'$, 
a continuous map $p : G' \to G$ which maps vertices to vertices and edges to edges
is called a \emph{covering} if $p$ satisfies the following two conditions:
\begin{itemize}
	\item[(i)]
	Induced map $V(p) : V(G') \to V(G)$ on the vertex set is surjective.
	\item[(ii)]
	Let $E(p) : E(G') \to E(G)$ be the induced map on the set of edges.
	For each vertex $v' \in V(G')$, the restriction of $E(p)$ on the set of edges with one endpoint on $v'$ is bijective.
\end{itemize}
If $G$ and $G'$ are metric graphs, we further require coverings to be local isometries.
\subsection{Train track representatives}
Let $F$ be a free group of rank $n \ge 2$, $\auto$ the group of automorphisms of $F$.
We denote by $i_\gamma : F \to F$ the inner automorphism of $\gamma \in F$, that is,
$i_\gamma(\delta) = \gamma \delta \gamma^{-1}$ for $\delta \in F$.
Let $\mathrm{Inn}(F)$ be the inner automorphism group.
The quotient 
\[
	\OUT= \auto / \mathrm{Inn}(F)
\]
is called the \emph{outer automorphism group} of $F$.
In what follows, by an outer automorphism, 
we mean an outer automorphism of a free group of rank $\geq 2$ unless otherwise stated.

As we need to consider outer automorphisms on several free groups, 
we prepare the following notation:
\begin{note}
For an outer automorphism $\phi$,
we use the notation $F(\phi)$ to denote the free group on which $\phi$ is defined. 
\end{note}

The group $\OUT$ acts on the set of conjugacy classes by 
the quotient of $\auto$ action on $F$.
An element $\phi \in \OUT$ is \emph{reducible} if 
there exists a non-trivial free decomposition $F = A_1 \ast A_2 \ast \cdots A_r \ast B$.
such that $\phi$ permutes the non-trivial conjugacy classes of $A_1,\ldots, A_r$.
An outer automorphism $\phi$ is \emph{irreducible} if it is not reducible.
If, in addition, $\phi^k$ ($k \ge 1$) are all irreducible then $\phi$ is
called \emph{fully irreducible}.

We fix an isomorphism $\pi_{1}(R_{n})\cong F$.
Then each loop in $R_{n}$ is labeled by an element of $F$.
Any homotopy equivalence $g$ between marked graphs determines
a homotopy equivalence between roses.
By considering the induced action on $\pi_{1}(R_{n})\cong F$,
we get an outer automorphism $\phi$ of $F$ via the marking.
Thus the map $g$ on a marked graph represents an outer automorphism $\phi$.
We always assume that $g$ is an immersion on each edge.
In addition, if $g$ maps vertices to vertices, 
we say that $g$ is a \emph{topological representative} of $\phi$.
A topological representative $g : G \to G$ on a metric graph $G$ is an \emph{affine train track map} if
\begin{itemize}
	\item for every edge $e \in E(G)$, $g\vert_{e}$ uniformly expands length by the stretch factor $\lambda>1$, 
	\item every power $g^k$ ($k \ge 1$) is an immersion on each edge, and
	\item for each pair of edges $e, e' \in E(G)$, we have $g^k(e) \supset e'$ for some $k \ge 1$.
\end{itemize}
It is known that every fully irreducible outer automorphism $\phi$
can be represented by an affine train track map (\cite[Theorem 1.7]{BH92}).
We call such a representative an \emph{affine train track representative}.
For a fully irreducible $\phi$, 
the {\em stretch factor} $\lambda>1$ can also be characterized as the exponential growth rate of 
the cyclic word length of any element in $F$ under iteration of any representative of $\phi$ \cite[Remark 1.8]{BH92}.
Hence the stretch factor $\lambda$ is independent of the choice of train track representatives, and denoted $\lambda(\phi)$.
Also it can be readily seen that $\lambda(\phi) = \lambda(\psi \phi \psi^{-1})$ for any $\psi \in \OUT$.
See also Proposition \ref{stability_of_limits} below.
Let $g : G \to G$ be an affine train track map.
We say that a path $\sigma$ in $G$ is \emph{legal} if $g(\sigma)_\sharp = g(\sigma)$.
In contrast, if $g^p(\sigma)_\sharp = \sigma$ for some $p\ge 1$,
the path $\sigma$ is called a \emph{periodic Nielsen path}. 
A \emph{Nielsen path} is a periodic Nielsen path with $p =1$. 
A periodic Nielsen path is {\em indivisible} if it is not a concatenation of non-trivial periodic Nielsen paths.
For a point $x \in G$, 
two geodesics $\gamma$ and $\gamma'$ emanating from $x$ determines 
the \emph{same direction} if $\gamma \cap \gamma' \neq \{x\}$.
This defines an equivalence relation on the set of arcs with an endpoint $x$ and
an equivalence class is called a \emph{direction} at $x$.
We denote by $D_{x}$ the set of directions at $x$.
These sets of directions play a similar role as tangent spaces.
Any homotopy equivalence $h:G\rightarrow G$ which maps vertices to vertices naturally induces a map, denoted $Dh$, on the set of directions.
A \emph{turn} is an unordered pair of directions with the same endpoint.
We may also define directions and turns on the universal cover $\wt G$ similarly.

\subsection{Periodicity of outer automorphisms}
The action of elements in $\auto$ on $F$ can be extended to a continuous action on the Gromov boundary $\partial F $ of $F$.
For $\Phi\in\auto$, we denote such an extension by $\hat{\Phi}: \partial F \to \partial F$.
Let $\phi \in \OUT$ be fully irreducible and $\Phi \in \auto$ its representative.  
We denote the fixed point set of the boundary extension by $\fix{\hat{\Phi}}$, 
and denote the subset of non-repelling fixed points by $\nfix{\hat{\Phi}}$. 
We say that $\Phi$ is a \emph{principal automorphism} if $\nfix{\hat{\Phi}}$ contains at least three points. 
The set of principal automorphisms representing $\phi$ is denoted by $\mathrm{PA}(\phi)$.
Let $g: G \to G$ be a train track representative of $\phi$ and $\wt G$ a universal covering of $G$.
There is a bijection between representatives of $\phi$ and the set of lifts of $g$ to $\wt G$.
The bijection is characterized by the twisted equivariance;
let $\breve{g}$ be a lift of $g$, then for any $\gamma\in F$, $\breve{g}\circ\gamma = \Phi(\gamma)\circ\breve{g}$ holds for some $\Phi\in\phi$.
A lift $\breve{g}:\wt G \to \wt G$ is called a principal lift if the corresponding automorphism is a principal automorphism.
A vertex $v \in V(G)$ is called a {\em principal vertex} if $v$ has at least three periodic directions
or $v$ is an endpoint of an indivisible periodic Nielsen path.
We call a lift $\wt{v}$ of a principal vertex $v$ also principal vertex on the universal covering.
One can see that there exists at least one principal vertex \cite[ Lemma 3.19]{FH11}.

A fully irreducible $\phi\in\OUT$ is called {\em rotationless} if for each $k>0$ and each $\Phi_k \in\mathrm{PA}(\phi^{k})$,
there is $\Phi\in\mathrm{PA}(\phi)$ such that $\Phi^{k} = \Phi_k$.


\subsection{Attracting tree, stable lamination and ideal Whitehead graph}\label{sec.TLW}
For each fully irreducible outer automorphism $\phi \in \OUT$, 
there exists an associated $\R$-tree $T_+(\phi)$ called the {\em attracting tree} for $\phi$.
We recall the construction of $T_+(\phi)$.
Let $g:G\to G$ be an affine train track representative of $\phi \in \OUT$.
Here the graph $G$ is marked by $\rho_G : R_n \to G$.
For each $i \ge 0$, the marking $g^i \circ\rho_G : R_n \to G$ gives a new marked graph which we denote by $G_i$.
Each $G_i$ is equipped with a metric so that each arc $c$ has length
\[
	l_{G_i}(c) = l_G(g^i(c)_\sharp)/\lambda^i.
\] 
The map $g:G\to G$ induces a homotopy equivalence denoted 
$g_{i+1,i} : G_i \to G_{i+1}$ that preserves marking,
namely, $g_{i+1,i} \circ \rho_G$ is homotopic to $g^i \circ\rho_G$. 
We fix a base point on $G_i$ so that the marking $g^i\circ\rho_G : R_r \to G_i$ preserves the base point. 
Then the maps $g_{j,i} = g_{j,j-1}\circ\cdots\circ g_{i+1,i} : G_i \to G_j$ also preserve the base point.
We lift the base point of each $G_i$ to the tree $\wt{G_i}$ and choose a lift $\wt{g_{j,i}}$ of $g_{j,i}$ 
via the base point.
Thus we get a direct system 
\[
	\wt{G} = \wt{G_0} \xrightarrow{\wt{g_{1,0}}} \wt{G_1} 
	\xrightarrow{\wt{g_{2,1}}} \cdots.
\]
Now we define the \emph{attracting tree} $T_+(\phi)$ for $\phi$ as the $F$-equivariant direct limit of 
the sequence $\{\wt{G_i}\}$.
The attracting tree is an $\R$-tree and well-defined up to isometric conjugacy.
In particular, it does not depend on the choice of the train track representatives.
If there is no confusion we simply denote $T_+(\phi)$ by $T_+$.
It is known that we have the direct limit map $f_g : \wt G \to T_+$ which is surjective,
$F$-equivariant and isometry on each legal path with respect to $g$.
Below we recall such basic properties of $T_{+}$ and $f_{g}$ from \cite{HM11}.

\begin{thm}[{\cite[Theorem 2.15.]{HM11}}]\label{associated map on T}
	Let $\phi \in \OUT$ be fully irreducible, and $T_+ = T_+(\phi)$. 
	To each $\Phi \in \auto$ representing $\phi$, there is associated a homothety 
	$\Phi_+: T_+ \to T_+$ with stretch factor $\lambda(\phi)$ satisfying the following properties:
	\begin{itemize}
	\item[$(1)$]
		For each affine train track representative $g : G \to G$ of $\phi$, letting
		$\breve{g} : \wt G\to \wt G$ be the lift corresponding to $\Phi$, 
		we have $\Phi_+ \circ f_g = f_g \circ \breve{g}$.
	\item[$(2)$] 
		For each $\Phi, \Phi'$ representing $\phi$, if $\gamma \in F$ is the unique element such that
		$\Phi' = i_\gamma \circ \Phi$ then $\Phi'_+ = \gamma \circ \Phi_+$.
	\end{itemize}
\end{thm}

\begin{lem}[{\cite[Lemma 2.16.]{HM11}}]\label{structure of T}
	Suppose that $\phi \in \outer$ is fully irreducible and rotationless,
	$g : G \to G$ is an affine train track representative, $\Phi$ is a principal automorphism representing
	$\phi$ and $\breve{g} : \wt{G} \to \wt{G}$ is the principle lift corresponding to $\Phi$, 
	and $\Phi_+ : T_+ \to T_+$ is as in Theorem \ref{associated map on T}. Then
	\begin{itemize}
	\item[$(1)$] 
		$f_g( \fix{\breve{g}})$ is a branch point $b$ of $T_+$, and $\fix{\Phi_+} = \{b\}$.
	\item[$(2)$] 
		Every direction based at $b$ is fixed by $D\Phi_+$ and has the form $Df_g(d)$ where
		$d$ is a fixed direction based at some $\wt{v} \in \fix{\breve{g}}$.
	\item[$(3)$]
		The assignment $\Phi \mapsto b = \fix{\Phi_+}$ of $(1)$ defines a bijection between
		$PA(\phi)$ and the set of branch points of $T_+$.
	\end{itemize}
\end{lem}

Let $g : G\to G$ be a train track representative of $\phi \in \out{r}$.
The \emph{stable lamination} $\Lambda(g)$ of $g$ is the set of all the bi-infinite edge paths
\[
	\gamma = \cdots e_{-1}e_0e_1e_2\cdots
\]
in the graph $G$ which satisfy the following;
for any $i,j \in \mathbb{Z}$ with $i \le j$, there exit $n \in \N$ and an edge $e \in E(G)$ such that 
the path $e_i \cdots e_j$ is a subpath of $g^n(e)$.
A path $\gamma \in \Lambda(g)$ is called a \emph{leaf}.
We also consider the lifts $\wt\Lambda(\phi)$ of leaves to the universal cover $\wt G$.
A leaf in $\wt\Lambda(\phi)$ determines an unordered pair of distinct points in the boundary $\partial\wt{G}$ of 
the universal cover of $G$.
Let $T$ be an $\mathbb{R}$-tree or a free group.
We define the {\em geodesic leaf space} $\mathcal{G}T$ by
$$\mathcal{G}T:=(\partial T\times \partial T\setminus\{\mathrm{diagonal}\})/(\mathbb{Z}/2)$$
where $\mathbb{Z}/2$ acts freely by permuting the coordinates.
By the above discussion, we may think $\wt\Lambda(\phi)$ as a subset of $\mathcal{G}\wt G$.
Since $\partial\wt{G}$ is naturally identified with $\partial F(\phi)$, we also regard $\wt\Lambda(g)$ as a subset of
$\mathcal{G}F(\phi)$.
It turns out $\Lambda(g)$ is independent of the choice of $g$.
Hence we denote it by $\Lambda(\phi)$,  and call  $\Lambda(\phi)$ the stable lamination of $\phi$.

We recall the work of Handel-Moser which explains the relation between the stable lamination and the attracting tree.
\begin{lem}[{\cite[Lemma 2.21]{HM11}}]\label{lem.realization}
For any train track representative $g:G\rightarrow G$ of $\phi$ and any leaf $\wt\ell$ of $\wt \Lambda(\phi)$,
the restriction of the map $f_{g}:\wt G\rightarrow T_{+}(\phi)$ to the realization $\wt\ell$ in $\wt G$ is an isometry.
Moreover, the bi-infinite line in $T_{+}(\phi)$ which is the image of this isometry is independent of the choice of $g$, depending only on $\wt\ell$;
this image is called the realization of $\wt\ell$ in $T_{+}(\phi)$
\end{lem}
By Lemma \ref{lem.realization}, we may also regard $\Lambda(\phi)$ as a subset of $\mathcal{G}T_{+}(\phi)$.

We recall the work of Bestvina-Feighn-Handel \cite{BFH97}.
They consider the space $\mathcal{IL}$ of stable laminations $\Lambda(\phi)$ 
as $\phi$ ranges over all irreducible outer automorphisms of a fixed free group $F$.
$\mathrm{Out}(F)$ acts on $\mathcal{IL}$ via $\psi\Lambda(\phi) := \Lambda(\psi\phi\psi^{-1})$.
For $\Lambda\in\mathcal{IL}$, we denote by $\mathrm{Stab}(\Lambda)$ the stabilizer with respect to this action.
The following proposition due to Bestvina-Feighn-Handel is relevant to our discussion of commensurability.
\begin{prop}[{\cite[Theorem 2.14 and Corollary 2.15]{BFH97}}]\label{prop.BFH2.16}
Let $\phi$ be a fully irreducible outer automorphism.
Then the stabilizer $\mathrm{Stab}(\Lambda(\phi))$ is virtually cyclic.
Furthermore, if $\lambda(\psi)\not = 1$ for some $\psi\in\mathrm{Stab}(\Lambda(\phi))$, then
$\phi$ and $\psi$ have common nonzero powers.
\end{prop}

For every principal automorphism $\Phi$, let $L(\Phi)$ be the set of leaves $\{P_1, P_2\}$ of $\Lambda(\phi)$ 
with $P_1, P_2 \in \nfix{\hat{\Phi}}$.
An element of $L(\Phi)$ is called {\em a singular leaf} associated to $\Phi$. 
We define \emph{the component of the ideal Whitehead graph determined by $\Phi$}, denoted $W(\Phi)$,
to be the graph with one vertex for each point in $\nfix{\hat{\Phi}}$, and two vertices corresponding to
$P_1, P_2 \in \nfix{\hat{\Phi}}$ are connected by an edge if there is a leaf $\{P_1, P_2\} \in L(\Phi)$. 
The \emph{ideal Whitehead graph} $\wt{\mathcal{IW}}(\phi)$ is defined as 
the disjoint union of components $W(\Phi)$, 
one for each principal automorphism $\Phi$ representing $\phi$.
$F(\phi)$ acts on $\wt{\mathcal{IW}}(\phi)$ isometrically and cofinitely.
We denote by $\mathcal{IW}(\phi)$ the quotient $\wt{\mathcal{IW}}(\phi)/F(\phi)$.
We remark that some components of $\wt{\mathcal{IW}}(\phi)$ may have vertices of valence 1.
For a component $W$ of $\wt{\mathcal{IW}}(\phi)$, we call a non-trivial graph automorphism on $W$ a {\em symmetry}.


Suppose that $b \in T_+$ is a branch point, 
$\Phi \in PA(\phi)$ is the principal automorphism corresponding to $b$ by Lemma \ref{structure of T}, 
and $\breve{g} : \wt{G} \to \wt{G}$ is the corresponding principal lift. 
Each singular leaf $\ell$ in $L(\Phi)$ passes through some 
$x \in \fix{\breve{g}}$, and so $f_g(\ell)$, 
the realization of $\ell$ in $T_+$, passes through the point $f_g(x) = b$.
We call these {\em the singular leaves at $b$}. 
Each singular leaf at $b$ is divided by $b$ into singular rays at $b$ 
which are the images under $f_g$ of the singular rays at some $x \in \fix{\breve{g}}$.
We define a graph $W(z;T_{+})$ as the graph with one vertex for each direction $d_i$ at $z$ and 
with an edge connecting $d_i$ to $d_j$ if the turn $(d_i,d_j)$ is taken by 
the realization of some leaf of $\Lambda_-$ in $T_+$. 
It turns out that the union of such $W(z;T_{+})$ over all branched points coincides with $\wt{\mathcal{IW}}(\phi)$ \cite[Section 3.3]{HM11}.
Finally, we note that for each $\wt x \in \wt G$ the map $D_{\wt x}f_g$ maps the directions of $\wt G$ at $\wt x$ to the directions of $T_+$ at $b$.

\section{Fibered commensurability}\label{sec.fcms}
\subsection{Definition and basic properties}
We first define a covering relation.
\begin{dfn}\label{defi.cover}
An outer automorphism $\phi_1$
 {\em covers} an outer automorphism $\phi_2$
if there exist 
\begin{itemize}
\item a finite index subgroup $H<F(\phi_2)$ which is isomorphic to $F(\phi_1)$,
\item representatives $\Phi_1\in\phi_1$, $\Phi_2\in\phi_2$, and
\item $k\in\mathbb{N}$
\end{itemize}
such that $\Phi_2^k(H) = H$ and $\Phi_2^k|_H = \Phi_1$.
\end{dfn}
\begin{rem}
The equality $\Phi_2^k|_H = \Phi_1$ must pass through the isomorphism between $H$ and $F(\phi_{1})$, 
however for notational simplicity we omit to write the isomorphism.
\end{rem}
\begin{rem}
Let $\phi$ and $\psi$ be fully irreducible.
If $\psi$ covers $\phi$ then there exist train track representatives 
$g' : G' \to G' , g : G\to G$ of $\psi$ and $\phi$ respectively such that
$G'$ covers $G$ and $g'$ is a lift of $g^{k}$ for some $k\in\mathbb{N}$.
\end{rem}
By using this covering relation, we define the following relation.
\begin{dfn}
Let $\phi_{1}$ and $\phi_{2}$ be outer automorphisms.
Then we say $\phi_{1}>\phi_{2}$ if there exists $k\in\mathbb{N}$
so that $\phi_{1}^{k}$ covers $\phi_{2}^{k}$.
Two outer automorphisms $\phi_1,\phi_2
$ are said to be {\em covering equivalent} if $\phi_{1}>\phi_{2}$ and 
$\phi_{2}>\phi_{1}$.
\end{dfn}

We first show that the relation $>$ is transitive and hence a total order on each commensurability class.
\begin{prop}\label{prop.covering-transitive}
	Let $\phi_1$, $\phi_2$ and $\phi_3$ be outer automorphisms. 
	Suppose $\phi_1>\phi_2$ and $\phi_2>\phi_3$, 
	then $\phi_1>\phi_3$.
\end{prop}
\begin{proof}
By definition, we see that if $\psi'$ covers $\psi$, then $(\psi')^{k}$ covers $\psi^{k}$.
Hence it suffices to prove that if $\phi_{1}$ covers $\phi_{2}$ and $\phi_{2}$ covers $\phi_{3}$, 
then  $\phi_{1}^{\ell}$ covers $\phi_{3}^{\ell}$ for some $\ell\geq 1$.
In this case we have finite index subgroups $F(\phi_{1})\cong H_{1}<F(\phi_{2})$ and $F(\phi_{2})\cong H_{2}<F(\phi_{3})$, and representatives $\Phi_{1}\in\phi_{1}$, $\Phi_{2},\Phi_{2}'\in\phi_{2}$, and $\Phi_{3}\in\phi_{3}$ such that 
\begin{itemize}
\item $\Phi_{2}(H_{1}) = H_{1}$ and $\Phi_{2}^{k_{1}}|_{H_{1}} = \Phi_{1}$, and
\item $\Phi_{3}(H_{2}) = H_{2}$ and $\Phi_{3}^{k_{2}}|_{H_{2}} = \Phi_{2}'$
\end{itemize}
	Note that if $\Phi'=i_{\eta}\Phi$, then $(\Phi')^{n}=i_{\eta_{n}}\Phi^{n}$ for $\eta_{n}=\eta\Phi(\eta)\Phi^{2}(\eta)\cdots \Phi^{n}(\eta)$.
	Since $H_{1}$ is of finite index, some power $(\Phi_{2}')^{\ell}$ satisfy $(\Phi_{2}')^{\ell}(H_{1}) = H_{1}$ and 
	$(\Phi_{2}')^{\ell} = i_{\delta}\Phi_{2}^{\ell}$ for some $\delta\in H_{1}$.
	Thus we see that $\Phi_{3}^{\ell k_{2}}(H_{1}) = H_{1}$ and  $\Phi_{3}^{\ell k_{1}k_{2}}|_{H_{1}} = (\Phi_{2}')^{\ell k_{1}}|_{H_{1}} = (\Phi_{1}')^{\ell}$ for some $\Phi_{1}'\in \phi_{1}$.
	Hence $\phi_{1}^{\ell}$ covers $\phi_{3}^{\ell}$.
\end{proof}

Note that if $\phi_1>\phi_2$ and $\phi_2>\phi_1$ then we have $F(\phi_{1}) = F(\phi_{2})$ 
by the rank of subgroups in the definition of coverings.
For rotationless fully irreducible case, the covering equivalence is no more than being conjugate.

\begin{prop}\label{prop.covering-equiv}
	Let $\phi_1, \phi_2$ be rotationless fully irreducible outer automorphisms with $F(\phi_{1}) = F(\phi_{2})=:F$.
	If $\phi_1$ and $\phi_2$ are covering equivalent, then there exist 
	$\Phi_1 \in \phi_1$, $\Phi_2 \in \phi_2$ and $\Phi_3 \in \mathrm{Aut}(F)$
	such that $\Phi_3^{-1}   \Phi_1   \Phi_3 = \Phi_2$.
\end{prop}

\begin{proof}
	Since $\phi_1^\ell$ covers $\phi_2^\ell$ for some $\ell \in \N$, 
	we have representatives $\Phi_{1,\ell} \in \phi_1^\ell$, 
	$\Phi_{2,\ell} \in \phi_2^\ell$ and an automorphism $\Psi \in \mathrm{Aut}(F)$ so that 
	$\Phi_{2,\ell}^{k} = \Psi^{-1}   \Phi_{1,\ell} \Psi$ 
	for some $k\in \N$.
	As $\Psi\circ i_{\gamma} = i_{\Psi(\gamma)}\circ\Psi$ and 
	principal automorphisms are defined by the number of non-repelling fixed points on $\partial F$,
	we may suppose both $\Phi_{2,\ell}^{k}$ and $\Phi_{1,\ell}$ are principal.
	Let $k' \in \N$ be the number in the definition that $\phi_2^{\ell'}$ covers $\phi_1^{\ell'}$ for some $\ell' \in \N$. 
	Let $\lambda_1$ and $\lambda_2$ are stretch factors of $\phi_1$ and $\phi_2$ respectively.
	Then we have $\lambda_1^{k'} = \lambda_2$ and $\lambda_2^k = \lambda_1$, and hence $k = k' =1$.
	Since $\phi_1$ and $\phi_2$ are rotationless, we can take $\Phi_1 \in \pa(\phi_1)$ and $\Phi_2 \in \pa(\phi_2)$
	so that $\Phi_i^\ell = \Phi_{i,\ell}$ and $\fixn(\Phi_i) = \fixn(\Phi_{i,\ell})$ ($i = 1,2$).
	Thus we have 
	\begin{equation}\label{eqn:1}
		\Phi_2^\ell = \Psi^{-1}   \Phi_1^\ell   \Psi \left(\iff \Phi_2^\ell = (\Psi^{-1}   \Phi_1   \Psi)^\ell\right).
	\end{equation}
  	On the other hand, for a fully irreducible outer automorphism $\phi$, 
	two elements $\Phi_1, \Phi_2 \in \pa(\phi)$ are distinct only if 
	$\fixn(\Phi_1) \cap \fixn(\Phi_2) = \emptyset$ (see \cite[Corollary 2.9]{HM11}).
	Together with this, the equation \eqref{eqn:1} gives 
	\[
		\Phi_2 = \Psi^{-1}   \Phi_1   \Psi.
	\]
	Letting $\Phi_3:= \Psi$ we have the conclusion.
\end{proof}

We use the following lemma at several places.
\begin{lem}\label{lem.auto-f.i.}
	Let $\Phi, \Psi$ be automorphisms on a free group $F$.
	If $\Phi$ and $\Psi$ coincide on a finite index subgroup $H < F$ and $\Phi(H) = \Psi(H) = H$
	then $\Phi = \Psi$. 
\end{lem}

\begin{proof}
	Recall that any finite index subgroup of a free group contains a normal subgroup of finite index. 
	Let $N$ be such a normal subgroup of $H$.
	Then $N$ is a free group of rank $\geq 2$.
	We consider representatives $a_1, \ldots, a_m$ of cosets of $H$ in $F$.
	It suffices to show that $\Phi(a_i) = \Psi(a_i)$ for each $i = 1,2,\ldots,m$.
	
	Fix $i \in \{1,2,\ldots,m\}$ and choose $x \in N < H$ arbitrary.
	By our assumption $\Phi(x) = \Psi(x) \in H$, we denote this element by $y$.
	Since $N$ is normal, $a_i x a_i^{-1}\in N \subset H$. Hence 
	\[
		\Psi(a_i) y \Psi(a_i)^{-1}=\Psi(a_i x a_i^{-1}) = \Phi(a_i x a_i^{-1}) = \Phi(a_i) y \Phi(a_i)^{-1}.
	\]
	This implies that $\Psi(a_i)^{-1}\Phi(a_i)$ commutes with any element in $\Phi(N) = \Psi(N)$.
	Since $F$ and $N$ are free, we deduce that $\Psi(a_i)^{-1}\Phi(a_i)$ is trivial.
\end{proof}

\begin{dfn}
Two outer automorphisms $\phi_1$ and $\phi_2$ are said to be {\em (fibered) commensurable}, denoted $\phi_1\sim\phi_2$,
 if there is a third outer automorphism $\phi_{3}$ such that $\phi_{3}>\phi_{1}$ and $\phi_{3}>\phi_{2}$.
\end{dfn}
In the next proposition, we justify the notation $\phi_1\sim\phi_2$.
\begin{prop}
Let $\phi_1$, $\phi_2$ and $\phi_3$ be outer automorphisms.
Suppose $\phi_1\sim\phi_2$ and $\phi_2\sim\phi_3$.
Then $\phi_1\sim\phi_3$.
\end{prop}
\begin{proof}
Let $\phi_{i,i+1}$ be an outer automorphism which satisfies $\phi_{i,i+1}>\phi_i$ and $\phi_{i,i+1}>\phi_{i+1}$ for $i=1,2$.
Then in $F(\phi_2)$, there are corresponding finite index subgroups $H_{1,2}$ and $H_{2,3}$ and representatives $\Phi_{1,2}$ and $\Phi_{2,3}$ of some power of $\phi_{2}$.
By the Nielsen-Schreier theorem, $H:=H_{1,2}\cap H_{2,3}$ is a finite index free subgroup.
Hence, there exist $k_{1,2},k_{2,3}\in\mathbb{N}$ such that
\begin{itemize}
\item there exists $\delta\in F(\phi_{2})$ such that $\Phi_{1,2}^{k_{1,2}} = i_{\delta} \Phi_{2,3}^{k_{2,3}}$, and
\item $\Phi_{1,2}^{k_{1,2}}(H) = H$ and $\Phi_{2,3}^{k_{2,3}}(H) = H$.
\end{itemize}
Hence $\phi_{1,2,3}:=[\Phi_{1,2}^{k_{1,2}}] = [\Phi_{2,3}^{k_{2,3}}] \in\mathrm{Out}(H)$ satisfies 
$\phi_{1,2,3}>\phi_{1,2}$ and $\phi_{1,2,3} >\phi_{2,3}$.
By Proposition \ref{prop.covering-transitive}, we see that $\phi_{1,2,3}>\phi_{1}$ and $\phi_{1,2,3}>\phi_{3}$.
\end{proof}

From now on we only consider covering equivalence classes and by abuse of notation, 
we simply denote by $\phi\in\out{n}$ the equivalence class.

Now we collect properties invariant under taking commensurability.
\begin{prop}\label{stability_of_limits}
Let $\phi$ and $\psi$ be fully irreducible outer automorphisms.
Suppose $\psi>\phi$ and we fix a copy of $F(\psi)$ in $F(\phi)$ given by the definition of the coverings.
Then,
\begin{enumerate}
\item there is an $F(\psi)$-equivariant isometry $f:T_{+}(\phi) \rightarrow T_{+}(\psi)$,
\item ideal Whitehead graphs are isomorphic as graphs, i.e. $\wt{\mathcal{IW}}(\phi) \cong \wt{\mathcal{IW}}(\psi)$,
\item $f(\wt\Lambda(\phi)) = \wt\Lambda(\psi)$ where the map $f$ here is an extension of the map $f$ in (1) to 
	$\mathcal{G}T_{+}(\phi)$, sill denoted by $f$ by abuse of notation. 
\item $\log(\lambda(\phi))/\log(\lambda(\psi))\in\mathbb{N}$.
\end{enumerate}
\end{prop}
\begin{proof}
	Let $g : G \to G$ and $h : G' \to G'$ be affine train track representatives of some powers $\phi^{\ell}$ and $\psi^{\ell}$ respectively such that
	there is a finite covering $p :G' \to G$ and $h : G' \to G'$ is a lift of $g^{k}$ for some $k\in\mathbb{N}$.
	Then we may identify universal coverings of $G$ and $G'$, and after taking iterations, $g$ and $h$ determine the same map on the universal covering.
	Since the attracting tree of $\phi$ is defined by the direct system constructed from the universal cover of $G$ and a lift of a power $g$,
	we see that $\phi$ and $\psi$ define the same attracting trees.
	In this case, the identification of $T_{+}(\phi)$ and $T_{+}(\psi)$ can only be made $F(\psi)$-equivariant.
	Thus (1) follows and we let $T_{+}:=T_{+}(\phi) = T_{+}(\psi)$.
	Note that $T_{+}$ only admit an action by $F(\psi)$.

	Then (2) follows from the characterization of the ideal Whitehead graph in terms of $T_{+}$, see \cite[Section 3.3]{HM11}.
	Now (3) follows as the stable laminations are subsets of $\mathcal{G}T_{+}$ defined by 
	$\Phi_{+}$ and $\Psi_{+}$ in Theorem \ref{associated map on T}, see also \cite[Section 2.8]{HM11}.
	
	For an affine train track representative $g : G\to G$ of a fully irreducible $\phi \in \mathrm{Out}(F(\phi))$,
	the logarithm of the expansion factor $\lambda(\phi)$ is equal to the expansion factor of edges of affine train track representatives.
	Let $\lambda(g)$ and $\lambda(h)$ denote the expansion factor of edges of affine train track representatives given above.
	Then we have $$\log\lambda(\psi) = \log\lambda(h) = \log\lambda(g^{k}) = k\log\lambda(g) = k\log\lambda(\phi).$$
\end{proof}

\subsection{Commensurability of atoroidal outer automorphisms}

An outer automorphism $\phi$ is \emph{toroidal} if there exists $k \in \N$ and a conjugacy class $[w]$ 
such that $\phi^k([w]) = [w]$.
If $\phi$ is not toroidal, then it is called \emph{atoroidal}.
Note that it is known (see \cite[Section 4]{BH92}) that if $\phi$ is fully irreducible and atoroidal, then $\phi$ is non-geometric.

\begin{prop}\label{prop.atroidal-iwip}
	For outer automorphisms $\phi$ and $\psi$ 
	assume that $\psi > \phi$.
	Then following holds:
	\begin{itemize}
	\item[$(1)$]
	$\phi$ is atoroidal $\iff$ $\psi$ is atoroidal.
	\item[$(2)$]
	$\phi$ is fully irreducible and atoroidal $\iff$ $\psi$ is fully irreducible and atoroidal.
	\end{itemize}
\end{prop}

\begin{proof}
	Let $H$ be a subgroup of $F(\phi)$ associated with $\psi$ which is isomorphic to $F(\psi)$.
	
	(1) If $\psi$ is toroidal, then some power of $\psi$ fixes a conjugacy class $[w]$ in $F(\psi)$.
	Since $\psi$ coincides with $\phi$ on a subgroup which includes $w$ up to isomorphism,
	$[w]$ must be fixed by some power of $\phi$ and hence $\phi$ is toroidal.
	
	Conversely, assume that $\phi$ is toroidal then there exists a conjugacy class $[w]$ in $F(\phi)$ 
	which is fixed by some power of $\phi$. 
	We may assume that $w^n$ for some $n$ is included in $H$ since $H$ is of finite index.
	Since a power of a representative of $\psi$ coincides with a representative of $\phi$ on $H$, 
	we have $\Psi(w^{n}) = a^{-1}w^{n}a$ for some $a\in F(\phi)$ and $\Psi\in\psi^{\ell}$.
	Then $\Psi^{k}(w^{n}) = (a\Psi(a)\cdots \Psi^{k}(a))^{-1}w^{n}a\Psi(a)\cdots \Psi^{k}(a)$.
	Hence for some $k$ we have $\psi([w^n]) = [w^{n}]$ in $H\cong F(\psi)$.
	This completes the proof of (1).
	
	(2)	
	The``if'' part follows from Kurosh Subgroup Theorem \cite[Theorem 1.10, Chapter IV]{LS}.
	The theorem states that for a free product $F = \Asterisk_i^m A_i \ast B$ of a group $F$, 
	any subgroup $F' < F$ is also represented as a free product
	$F' = \bigast_i^m C_i \ast D$ where each $C_i$ is the intersection of $F'$ with a conjugate of $A_i$.
	Suppose there exists a proper free factor $A$ of $F(\phi)$ whose conjugacy class is preserved by a power of $\phi$.
	Then similarly to the case of (1), there also exists a proper free factor of $H$ (which is the intersection of $H$ with a conjugate of $A$) 
	whose conjugacy class is preserved by a power of $\psi$.
	Hence the fully irreducibility of $\psi$ implies the fully irreducibility of $\phi$.
	 
	The converse is proved by referring Bestvina-Feighn-Handel \cite{BFH97} and Kapovich \cite{Kap14}. 
	Note that in \cite{Kap14}, fully irreducible elements are called iwip.
	Suppose $\phi$ is atoroidal and fully irreducible and $g:G\rightarrow G$ is an affine train track representative.
	In \cite{Kap14}, a train track representative is called {\em clean} if its transition matrix is positive after iteration, and its Whitehead graph is connected 	(we omit the definitions here, see \cite{Kap14} for details).
	In \cite[Proposition 4.4]{Kap14}, it is proved that for an atoroidal element,
	being fully irreducible and having clean train track representative are equivalent.
	Hence we may suppose $g:G\rightarrow G$ is clean.
	Also, by \cite[Lemma 2.1]{BFH97}, any lift $g':G'\rightarrow G'$ corresponding to $F(\psi)<F(\phi)$ is also clean.
	As we know that $\psi$ is atoroidal by (1), again by \cite[Proposition 4.4]{Kap14} we see that $\psi$ is fully irreducible.
\end{proof}


Although it is not directly related to our main theorem,
we discuss ageometric outer automorphisms.

\begin{dfn}
	A fully irreducible outer automorphism $\phi \in \out{n}$ is \emph{ageometric}
	if there exists a train track representative $g : G \to G$ which does not have any periodic Nielsen path.
	We also call such a train track representative without periodic Nielsen path {\em ageometric}.
\end{dfn}

We remark that if an outer automorphism is ageometric then it is atoroidal.
Indeed, if a conjugacy class $[\gamma]$ is preserved by $\phi$, 
then for any train track representative $g : G \to G$ of $\phi$, 
any path written as $\sigma_0 \gamma \sigma_0^{-1}$ is mapped to a path written as $\sigma_1 \gamma \sigma_1^{-1}$.
Thus, in particular, we have $g(\gamma)_\sharp = \gamma$.
This means that the path $\gamma$ is a Nielsen path and hence $\phi$ is not ageometric.

We show that the ageometricity is a commensurability invariant.
To do this, we recall an equivalent definition of ageometricity.
We first recall the notion of geometric index.
Let $T_{+}(\phi)$ be the attracting tree of an atoroidal fully irreducible $\phi\in\mathrm{Out}(F(\phi))$.
For each branched point $b$ of $T_{+}(\phi)$, the degree of $b$ denoted $\mathrm{deg}(b)$ is the number of components of $T_{+}(\phi)\setminus \{b\}$.
For each $b$, let $[b]$ denote the $F(\phi)$ equivalence class.
Then the geometric index of $T_{+}(\phi)$ is
$$\mathrm{ind}_{\mathrm{geod}} (T_{+}(\phi)) := \sum_{[b]:\mathrm{deg}(b)\geq 3}(\mathrm{deg}(b)-2).$$
Then it is known \cite{BF} that atoroidal fully irreducible $\phi$ is ageometric 
if and only if $\mathrm{ind}_{\mathrm{geod}} (T_{+})(\phi) < 2\cdot\mathrm{rank}(F(\phi)) -2$.
\begin{prop}
Let $\phi$ and $\psi$ be atoroidal fully irreducible automorphisms.
Suppose $\psi>\phi$.
Then $\phi$ is ageometric if and only if $\psi$ is ageometric.
\end{prop}
\begin{proof}
By Proposition \ref{stability_of_limits}, we see that $T_{+}(\phi) = T_{+}(\psi)$.
If $F(\psi)$ is a subgroup of $F(\phi)$ of index $m$, then 
\begin{eqnarray}
\mathrm{ind}_{\mathrm{geod}} (T_{+})(\psi) = m\cdot \mathrm{ind}_{\mathrm{geod}} (T_{+})(\phi),\\
 \left(\mathrm{rank}(F(\psi)) -1\right) = m\cdot \left(\mathrm{rank}(F(\phi)) -1\right)
\end{eqnarray}
Hence 
$$\mathrm{ind}_{\mathrm{geod}} (T_{+}(\phi)) < 2\cdot\mathrm{rank}(F(\phi)) -2\iff \mathrm{ind}_{\mathrm{geod}} (T_{+}(\psi)) < 2\cdot\mathrm{rank}(F(\psi)) -2.$$
\end{proof}


\section{Minimal elements}\label{sec.minimal}
\subsection{Topologically minimal element}
By Proposition \ref{prop.atroidal-iwip}, we see that being atoroidal and fully irreducible is a commensurability invariant.
In this subsection, for atoroidal and fully irreducible case, 
we prove that there is a minimum of the rank of free groups in any commensurability class under asymmetry condition on the 
ideal Whitehead graphs.
We call such an outer automorphism defined on the minimum rank free group, a {\em topologically minimal element}.
\begin{lem}\label{lem.top-min}
Let $\phi$ be an atoroidal fully irreducible outer automorphism and 
$[\phi]$ the fibered commensurability class of $\phi$.
Suppose that every component of the ideal Whitehead graph $\wt{\mathcal{IW}}(\phi)$ admits no symmetry.
Then there is an element $\phi_{m}\in[\phi]$ such that for any $\phi'\in[\phi]$, $F(\phi')$ is a finite index subgroup of 
$F(\phi_{m})$.
\end{lem}
\begin{proof}
By Proposition \ref{stability_of_limits},
for any $\psi\in[\phi]$, we have 
$\wt{\mathcal{IW}}(\phi) = \wt{\mathcal{IW}}(\psi)=:\wt{\mathcal{IW}}$ and $T_{+}(\phi) = T_{+}(\psi)=:T_{+}$.
Let $\psi\in[\phi]$.
Both $F(\phi)$ and $F(\psi)$ act on $T_{+}$ isometrically.
We first show that in the group of orientation preserving isometries $\mathrm{Isom}^{+}(T_{+})$, 
$\langle F(\phi)\cup F(\psi) \rangle$ contains $F(\phi)$ and $F(\psi)$ as finite index subgroups.
First, we consider the action of $\langle F(\psi)\cup F(\phi) \rangle$ on $\wt{\mathcal{IW}}$.
Note that $F(\phi)$ and $F(\psi)$ act on $\wt{\mathcal{IW}}$ so that 
the stabilizer of each component is trivial and their quotients have finitely many components.
Suppose that there is $\gamma\in\langle F(\psi)\cup F(\phi) \rangle$ which stabilizes a component.
As we have assumed that each component of $\wt{\mathcal{IW}}$ admits no symmetry, 
$\gamma$ must be trivial on the component.
Now we regard $\gamma$ as an isometry on $T_{+}$.
Fixing an component of $\wt{\mathcal{IW}}$ means that $\gamma$ fix a branched point $b$ of $T_{+}$, and 
$\gamma$ acts trivially on the component means that $D\gamma$ fix every direction from $b$.
Hence $\gamma$ must be the identity.
Let $\gamma' F(\phi)$ be a coset in $\langle F(\psi)\cup F(\phi) \rangle$ of $F(\phi)$.
If $\gamma'$ fix some component of $\wt{\mathcal{IW}}/F(\phi)$,
then we see that $\gamma'\delta$ for some $\delta\in F(\phi)$ fixes a component of $\wt{\mathcal{IW}}$.
But this means that $\gamma'\delta$ is identity and hence $\gamma'\in F(\phi)$.
Hence $\gamma'$ fixes no component of $\wt{\mathcal{IW}}/F(\phi)$.
As $\wt{\mathcal{IW}}/F(\phi)$ has only finitely many components, 
we see that $F(\phi)$ is a finite index subgroup of $\langle F(\psi)\cup F(\phi) \rangle$.
By the same argument, we also conclude that $F(\psi)$ is of finite index in $\langle F(\psi)\cup F(\phi) \rangle$.

Recall that any element acting on an $\mathbb{R}$-tree by a finite order must fix a point.
Since every element of $\langle F(\psi)\cup F(\phi) \rangle$ is orientation preserving,
if an element of $\langle F(\psi)\cup F(\phi) \rangle$ fixes a point which is not a branched point, then it must fix some branched point as well.
But we know that there is no non-trivial stabilizer of any branched point.
Therefore we see that $\langle F(\psi)\cup F(\phi) \rangle$ is torsion free and hence $\langle F(\psi)\cup F(\phi) \rangle$ is a free group by \cite[Theorem 0.2]{Sta}.
Since $\phi$ and $\psi$ are commensurable, we have representatives $\Phi,\Psi$ of $\phi,\psi$ respectively and a finite index subgroup $H$ of $F(\phi)$ and $F(\psi)$ such that $\Phi^{n}|_{H} = \Psi^{m}|_{H}$ for some $n,m\in \mathbb{N}$.
Let $\Theta_{+}:T_{+}\rightarrow T_{+}$ be a map induced from $\Phi^{n}|_{H} = \Psi^{m}|_{H}$ by Theorem \ref{associated map on T}.
We set $\Theta(\gamma) = \Phi^{n}(\gamma)$ if $\gamma\in F(\phi)$ and $\Theta(\gamma) = \Psi^{m}(\gamma)$ if $\gamma\in F(\psi)$.
Then we have $\Theta_{+}\circ\gamma = \Theta(\gamma)\circ\Theta_{+}$ for any $\gamma\in F(\phi)\cup F(\psi)$.
Let $\gamma_{1}\cdots\gamma_{n}$ be a concatenation of elements where $\gamma_{i}\in F(\phi)$ if $i$ even and $\gamma_{i}\in F(\psi)$ otherwise.
We define $$\Theta(\gamma_{1}\cdots\gamma_{n}):=\Theta(\gamma_{1})\cdots \Theta(\gamma_{n}).$$
By the above discussion we have 
\begin{equation}\label{eq.Theta}
\Theta_{+}\circ\gamma_{1}\cdots\gamma_{n} = \Theta(\gamma_{1}\cdots\gamma_{n})\circ\Theta_{+}.
\end{equation}

As the stabilizer of any branched point is trivial, we see that
if two concatenations $\gamma_{1}\cdots\gamma_{n}$ and $\gamma'_{1}\cdots\gamma'_{n}$ of elements of $F(\phi)$ and $F(\psi)$ 
represents the same element in $\mathrm{Isom}^{+}(T_{+})$,
the equation (\ref{eq.Theta}) implies
 $\Theta(\gamma_{1}\cdots\gamma_{n}) = \Theta(\gamma'_{1}\cdots\gamma'_{n})$ in $\mathrm{isom}^{+}(T_{+})$.
Therefore $\Theta$ is a well-defined homomorphism on $\langle F(\psi)\cup F(\phi) \rangle$ and
\begin{equation}\label{eq.Theta2}
\Theta|_{F(\phi)} = \Phi^{n} \text { and }\Theta|_{F(\psi)} = \Psi^{m}.
\end{equation}
By (\ref{eq.Theta}), we see that $\Theta$ is injective and since $\Phi^{n}$ and $\Psi^{m}$ are automorphisms,
it follows from (\ref{eq.Theta2}) that $\Theta$ is surjective.
Hence we have $\Theta\in\mathrm{Aut}(\langle F(\psi)\cup F(\phi) \rangle)$.
Moreover, the outer automorphism which is represented by $\Theta$ on $\langle F(\psi)\cup F(\phi) \rangle$ is an element of $[\phi]$ by (5).
Note that the rank of $\langle F(\psi)\cup F(\phi) \rangle$ is strictly smaller than that of $F(\phi)$ or $F(\psi)$, 
unless $\phi>\psi$ or $\psi>\phi$.
Therefore, by repeating the procedure above, we obtain an element $\phi_{m}\in[\phi]$ such that 
$F(\phi_{m})$ contains $F(\phi')$ as a finite index subgroup for any $\phi'\in[\phi]$.
\end{proof}
\begin{cor}\label{cor.covered by all}
Suppose the same assumption as Lemma \ref{lem.top-min}, and let $\phi_{m}$ be an element given by Lemma \ref{lem.top-min}.
Then every element $\psi\in[\phi]$ has a positive power such that $\psi^{n}>\phi_{m}$.
\end{cor}
\begin{proof}
Note that by Lemma \ref{lem.top-min}, we may suppose $F(\psi)$ is a finite index subgroup of $F(\phi_{m})$.
Since $\psi\sim\phi_{m}$, we see that there are representatives $\Psi$ and $\Phi_{m}$ of some powers of $\psi$ and $\phi_{m}$ which agree in a finite index subgroup of $F(\psi)$.
After taking higher powers if necessary, we may suppose that $\Phi_{m}$ preserves $F(\psi)$ in $F(\phi_{m})$.
Then by Lemma \ref{lem.auto-f.i.}, we see that $\Phi_{m}|_{F(\psi)} = \Psi$.
\end{proof}
\subsection{Dynamically minimal element}
\begin{lem}\label{lem.dynamical}
Let $\psi$ and $\phi$ be atoroidal and fully irreducible.
Suppose that $\psi^{n}$ covers $\phi$.
Suppose further that every component of the ideal Whitehead graph $\wt{\mathcal{IW}}:=\wt{\mathcal{IW}}(\phi)=\wt{\mathcal{IW}}(\psi)$ admits no symmetry.
Then there is an element $\psi'$ in the commensurability class $[\psi]$ such that $\lambda(\psi')=\lambda(\psi)$ and 
both $\psi$ and $\phi$ covers $\psi'$.
In particular, if $\phi = \phi_{m}$, the topologically minimal element given by Lemma \ref{lem.top-min}, 
then we have $F(\phi_{m}) = F(\psi')$.
\end{lem}
\begin{proof}
By the definition  of covering, there are representations $\Psi\in\psi$ and $\Phi\in\phi$ such that 
\begin{equation}\label{eq. covering}
\Phi|_{F(\psi)} = i_{\delta}\circ\Psi^{n}
\end{equation}
for some $\delta\in F(\psi)$.
Let $T_{+}:=T_{+}(\psi) = T_{+}(\phi)$ (Proposition \ref{stability_of_limits}) and
$\Psi_{+}$ be a homothety given by Theorem \ref{associated map on T}.
We define a map $\Psi_{\sharp}:\IT\rightarrow\IT$ by
$$\Psi_{\sharp}(\gamma):=\Psi_{+}\circ\gamma\circ\Psi_{+}^{-1}.$$
To see this is well-defined, we need to prove $\Psi_{+}\circ\gamma\circ\Psi_{+}^{-1}\in\IT$.
First, note that  $\Psi_{+}$ is an injection as it is a homothety, and is surjective by its construction.
(see Theorem \ref{associated map on T} and Lemma \ref{structure of T}).
Hence $\Psi^{-1}_{+}$ makes sense.
Given two points $x,y\in T_{+}$, we have 
\begin{align*}
&d(\Psi_{+}\circ\gamma\circ\Psi_{+}^{-1}(x),\Psi_{+}\circ\gamma\circ\Psi_{+}^{-1}(y))\\
=&\lambda(\psi) d(\gamma\circ\Psi_{+}^{-1}(x),\gamma\circ\Psi_{+}^{-1}(y))\\
=&\lambda(\psi) d(\Psi_{+}^{-1}(x),\Psi_{+}^{-1}(y))\\
=&\lambda(\psi) \cdot \frac{1}{\lambda(\psi)}d(x,y)\\
=&d(x,y).
\end{align*}
Therefore, $\Psi_{\sharp}$ is well-defined.
Moreover, if for any $x\in T_{+}$ we have
$\Psi_{+}\circ\gamma\circ\Psi_{+}^{-1}(x) = x$, then
$\gamma\circ\Psi_{+}^{-1}(x) = \Psi_{+}^{-1}(x)$.
Hence if $\Psi_{+}\circ\gamma\circ\Psi_{+}^{-1}$ is the identity, then $\gamma$ must be the identity.
Thus, we see that $\Psi_{\sharp}$ is injective.
Furthermore, $\Psi_{\sharp}$ is a homomorphism as $\Psi_{+}$ is a bijection.
Note that we have $F(\psi)<F(\phi)<\IT$ and $\Psi_{\sharp}(F(\psi)) = F(\psi)$ in $\IT$.
By (\ref{eq. covering}) and Theorem \ref{associated map on T},
we have 
\begin{equation}
\Psi_{\sharp}^{n}(F(\phi)) = F(\phi).\label{eq.psi}
\end{equation}

We now consider $$\bar F:=\langle F(\phi)\cup\Psi_{\sharp}(F(\phi))\cup \Psi_{\sharp}^{2}(F(\phi))\cup\cdots \cup \Psi_{\sharp}^{n-1}(F(\phi))\rangle.$$
By (\ref{eq.psi}), $\Psi_{\sharp}(\bar F) = \bar F$.
Also, $\Psi_{\sharp}^{k}(F(\phi))$ contains $F(\psi)$ as a finite index subgroup for any $k\in\mathbb{N}$.
As $\Psi_{+}$ induces a graph automorphism on $\wt{\mathcal{IW}}$,
we see that every element $\gamma\in\bar F$ induces a graph automorphism on $\wt{\mathcal{IW}}$.
Since we assumed that every component of $\wt{\mathcal{IW}}$ is asymmetric and $F(\psi)<\bar F$, 
$\wt{\mathcal{IW}}/\bar F$ has only finitely many components.
Hence similarly to the discussion for Lemma \ref{lem.top-min}, we see that $F(\psi)$ is a finite index subgroup of $\bar F$ and $\bar F$ is a free group.
Hence $\Psi_{\sharp}$ gives an automorphism on $\bar F$ whose restriction on $F(\psi)$ is $\Psi$.
Let $\psi'\in\mathrm{Out}(\bar F)$ denote the outer automorphism containing $\Psi_{\sharp}$.
Then $\lambda(\psi)=\lambda(\psi')$ and $\psi$ and $\phi$ covers $\psi'$.
If $\phi=\phi_{m}$ is a topologically minimal element, then we must have $F(\phi_{m}) = \bar F = F(\psi')$.
\end{proof}

Now we are ready to prove the main theorem.
\begin{proof}[Proof of Theorem \ref{thm.main}]
Let $\phi_{m}\in[\phi]$ be a topologically minimal element given by Lemma \ref{lem.top-min}.
Let $\phi_{1}, \phi_{2}\in[\phi]$ with $F(\phi_{1})\cong F(\phi_{2}) \cong F(\phi_{m})$.
By Proposition \ref{stability_of_limits}, we see that $\log(\lambda(\phi_{2}))/\log(\lambda(\phi_{1}))\in\mathbb{Q}_{+}$.
Without loss of generality, we may assume $\lambda(\phi_{2})>\lambda(\phi_{1})$.
Suppose $\log(\lambda(\phi_{2}))/\log(\lambda(\phi_{1}))\in\mathbb{Q}_{+}\setminus\mathbb{N}$.
Then there is some integers $m,n\in\mathbb{Z}$ such that 
$\lambda:=\log(\lambda(\phi_{1}^{m}))+\log(\lambda(\phi_{2}^{n}))$ is the greatest common divisor, 
namely $\log(\lambda(\phi_{i}))/\lambda\in\mathbb{N}$ for both $i=1,2$ and $\lambda$ is greatest among real numbers with this property.
Without loss of generality, we may assume $m>0$ and $n<0$.
By Proposition \ref{stability_of_limits}) and $F(\phi_{1})\cong F(\phi_{2})$, we may suppose
$\Lambda(\phi_{1})=\Lambda(\phi_{2})$.
Let $\Lambda:=\Lambda(\phi_{1})=\Lambda(\phi_{2})$ 
denote the stable lamination.
Now since both $\phi_{1}$ and $\phi_{2}$ are in $\mathrm{Stab}(\Lambda)$,
$\phi_{2}^{n}\phi_{1}^{m}$ is in $\mathrm{Stab}(\Lambda)$ with stretch factor greater than $1$.
Hence by Proposition \ref{prop.BFH2.16}, we see that  $\phi':=\phi_{2}^{n}\phi_{1}^{m}\in [\phi]$.
Thus we can find an element in $[\phi]$ whose logarithm of the stretch factor is at most half of those of $\phi_{1}$ and $\phi_{2}$.
Now suppose there are $\psi\in[\phi]$ with  $\lambda(\phi')>\lambda(\psi)$. 
By combining Corollary \ref{cor.covered by all} and Lemma \ref{lem.dynamical}, 
we may find $\psi'\in[\phi]$ with the same stretch factor as $\psi$ and $F(\psi')\cong F(\phi_{m})$.
Then by repeating the argument above, we may find an element in $[\phi]$ 
whose logarithm of the stretch factor is at most half of that of $\phi'$.
Since there is a lower bound of the smallest stretch factor of fully irreducible elements in $\mathrm{Out}(F(\phi_{m}))$,
this process would terminate.
Thus we can find $\phi_{\min}$ with $\log(\lambda(\varphi))/\log(\lambda(\phi_{\min}))\in\mathbb{N}$ for all $\varphi\in[\phi]$ and 
$F(\phi_{\min})\cong F(\phi_{m})$.
If there is another element with the same property, again by Proposition \ref{prop.BFH2.16} they are covering equivalent.
Therefore $\phi_{\min}$ is the unique (covering equivalence class of) minimal element with respect to $<$.
\end{proof}

\section*{Acknowledgement}
The authors would like to thank an anonymous referee for useful comments.


\end{document}